\documentclass[a4paper,11pt]{article}
\usepackage[english]{babel}
\usepackage[latin1]{inputenc}
\usepackage[T1]{fontenc}
\usepackage{amsthm}
\usepackage{amsmath}
\usepackage{amssymb}
\usepackage{xypic}
\usepackage{mathrsfs}
\usepackage{braket}
\usepackage{xcolor}
\usepackage{hyperref}
\usepackage{graphicx,subfig,floatrow}
\usepackage[nottoc]{tocbibind}
\usepackage{caption}
\usepackage{tikz}
\usetikzlibrary{arrows,shapes,positioning,mindmap,trees,automata,decorations.markings}
\usepackage{arydshln}
\usepackage{mathtools}

\hypersetup{
	colorlinks=false,
	allbordercolors=white,
}

\theoremstyle{definition}
\newtheorem{definition}{Definition}[section]

\theoremstyle{plain}
\newtheorem{teo}[definition]{Theorem}
\newtheorem{prop}[definition]{Proposition}
\newtheorem{cor}[definition]{Corollary}

\newtheorem{prob}[definition]{Problem}
\newtheorem{rem}[definition]{Remark}

\title{Graphs with large palette index }

\author{D.Mattiolo\thanks{Partially supported by Fondazione Cariverona, program ``Ricerca Scientifica di Eccellenza 2018'', project ``Reducing complexity in algebra, logic, combinatorics - REDCOM''}, G.Mazzuoccolo \thanks{Dipartimento di Informatica,
Universit\`a di Verona, Strada Le Grazie 15, Verona, Italy.}, G.Tabarelli \thanks{Dipartimento di Matematica,
Universit\`a di Trento, Via Sommarive 14, Povo (TN), Italy.}
}

\begin{document}
	\maketitle

	\begin{abstract}
	 Given an edge-coloring of a graph, the palette of a vertex is defined as the set of colors of the edges which are incident with it. We define the palette index of a graph as the minimum number of distinct palettes, taken over all edge-colorings, occurring among the vertices of the graph.
	 Several results about the palette index of some specific classes of graphs are known. In this paper we propose a different approach that leads to new and more general results on the palette index. Our main theorem gives a sufficient condition for a graph to have palette index larger than its minimum degree. In the second part of the paper, by using such a result, we answer to two open problems on this topic. 
	 First, for every $r$ odd, we construct a family of $r$-regular graphs with palette index reaching the maximum admissible value. After that, we construct the first known family of simple graphs whose palette index grows quadratically with respect to their maximum degree.
	\end{abstract}

	\section{Introduction and definitions}
	
	In this paper we study the palette index of graphs, which is a parameter related to colorings of the edges. In this context a $k$-edge-coloring of a graph $G$ is a map $c\colon E(G)\to\{1,\dots, k\}$ such that, for every pair of incident edges $e_1,e_2\in E(G)$, $c(e_1)\ne c(e_2)$. We define the \emph{palette} of a vertex $v\in V(G)$, with respect to the edge-coloring $c$ of $G$, to be the set $P_c(v) = \{c(e)\colon e\in E(G) $ and $e$ is incident to $v\}$. 
	The \emph{palette index} $\check s(G)$ of a graph $G$ is the minimum number of distinct palettes, taken over all edge-colorings, occurring among the vertices of the graph. 
	This parameter was formally introduced in \cite{meszka} and several results have appeared since then, see \cite{BonBonMaz, BonMaz, CasPet, Gha, HorHud, Smb}. All mentioned papers mainly consider the computation of the palette index in some special classes of graphs, such as trees, complete graphs, complete bipartite graphs, $3-$ and $4-$regular graphs and some others.
	Furthermore, it turns out that the palette index can be used to model some problems related to the self-assembly of DNA structure, see \cite{BonFer}. 
	
	It is an easy consequence of the definition that a graph has palette index equal to $1$ if and only if it is a $k$-regular graph admitting a $k$-edge-coloring. Moreover, it is proved in \cite{meszka} that no regular graph has palette index equal to $2$.
	The situation is less understood when we ask for $r$-regular graphs with a large palette index.
	The case of cubic graphs (i.e.\ $r=3$) is completely solved by the following theorem.

	\begin{teo}[\cite{meszka}]\label{cubic}
		Let $G$ be a connected cubic graph. \begin{itemize}
			\item $G$ is $3$-edge-colorable if and only if $\check s(G)=1$;
			\item $G$ is not $3$-edge-colorable with a $1$-factor if and only if $\check s(G)=3$;
			\item $G$ is not $3$-edge-colorable without a $1$-factor if and only if $\check s(G)=4$.
		\end{itemize}
	\end{teo}
	
	It is a trivial observation that $\check s(G) \leqslant r+1$ for an $r$-regular graph $G$: indeed, every $(r+1)$-edge-coloring (such a coloring does exist by a classical theorem of Vizing \cite{Viz}) has at most $r+1$ distinct palettes.
    Theorem \ref{cubic} shows an infinite family of cubic graphs with palette index $4$, and an example of a $4$-regular graph with palette index $5$ is constructed in \cite{BonMaz}.
    
    The general question about the existence of a $r$-regular graph $G$ with the maximum possible palette index, i.e.\ $\check s(G)=r+1$, is open for every $r>4$. In the present paper we give a complete answer to this question for every $r$ odd. 
    The main contribute of this paper is Theorem \ref{main_theorem}, a general result which gives a sufficient condition for an arbitrary graph $G$ to have $\check s(G)$ larger than its minimum degree. 
    In the special case that $G$ is an odd regular graph, Theorem \ref{main_theorem} furnishes a very strong restriction for the value of the palette index of $G$, see Corollary \ref{cor:regular}. 
    In Section \ref{sec:regular} we make use of such result to construct, for every positive integer $k$, a family of $(2k+1)$-regular graphs with $\check s(G)=2k+2$. 
	
	Finally, in Section \ref{sec:quadratic}, we move our attention to the non-regular case. As far as we know, no family of graphs with palette index which grows faster than $\Delta \log(\Delta)$  is known, where $\Delta$ denotes the maximum degree of $G$.
	In \cite{Aves_Bonis_Mazz}, a family of multigraphs whose palette index is expressed by a quadratic polynomial in $\Delta$ is presented. Problem 5.1 in \cite{Aves_Bonis_Mazz} asks for the construction of a family of graphs without multiple edges with the same property.  
	Here we give a complete answer to such a problem. Indeed, as a by-product of our result for odd regular graphs, we show a family of graphs having palette index which grows quadratically in $\Delta$.
	
\section{Graphs with palette index larger than their minimum degree}

In this section we prove a sufficient condition for a graph to have palette index larger than its minimum degree. Before going to the main result we give some preliminary definitions. 
Let $G$ be a graph, we denote by $\Delta(G)$ and $\delta(G)$ the maximum and minimum degree of $G$, respectively. Moreover, for every vertex $v$ of $G$ we denote by $d_G(v)$ the degree of $v$ in $G$.

A subgraph $K$ of $G$ is an \emph{even subgraph} of $G$ if every vertex has even degree in $K$. A \emph{spanning even subgraph} of $G$ is an even subgraph $K$ of G such that $V(K)=V(G)$.

Let $c$ be an edge-coloring of a graph $G$ with colors in the set $\cal C$ and generating palettes $P_1,\dots, P_t$. 
We consider a map $\phi_c\colon{\cal P}({\cal C}) \to (\mathbb{Z}_2^t,+) $, associated to $c$, where ${\cal P}({\cal C})$ denotes the power set of $\cal C$ and $(\mathbb{Z}_2^t,+)$ denotes the elementary abelian group of order $2^t$ whose elements are all strings of length $t$ with values $0$ and $1$\textcolor{blue}{,} and $+$ denotes the addition modulo $2$. For every subset $A$ of $\cal C$ we define $\phi_c(A)=(p_1,\dots,p_t)$, where we set

\[ 
p_i= \begin{cases} 0 & \text{ if $|P_i \cap A|$ even,} \\
        1 & \text{ if $|P_i \cap A|$ odd,} \\
       \end{cases}
\] 
for every $i\in\{1,\dots,t\}$.

In other words, the map $\phi_c$ establishes the parity of the cardinality of the intersection between $A$ and every palette of $c$.

For every subset $A$ of $\cal C$ we consider the subgraph $G_A$ induced by all edges with a color in $A$ in the edge-coloring $c$.  
The following remark is straightforward.

\begin{rem}\label{phi_parity}
Let $c$ be a proper edge-coloring of $G$ and let $A$ be a subset of the set of colors. The subgraph $G_A$ of $G$ is an even subgraph of $G$ if and only if $\phi_c(A)=(0,\dots,0)$.
\end{rem}

We are now in position to prove our main result.

\begin{teo}\label{main_theorem}
Let $G$ be a graph such that $\Delta(G) \geqslant 2$ and $G$ has no spanning even subgraph without isolated vertices. Then, $\check s(G) > \delta(G) $.
\end{teo}
\begin{proof}

First of all, we observe that the number of colors in any edge-coloring of $G$ is larger than $\delta(G)$. If $G$ is not regular, this follows immediately since the number of colors is at least $\Delta(G)>\delta(G)$. If $G$ is a $r$-regular graph with $r>1$, then $r$ is odd, otherwise $G$ itself is a spanning even subgraph without isolated vertices.
In this case $G$ has not a perfect matching $M$, otherwise the complement of $M$ would be a spanning even subgraph of minimum degree $r-1$, a contradiction.
Then, $G$ cannot admit an $r$-edge-coloring.

By contradiction we assume that $G$ admits an edge-coloring $c$ with colors in $\cal C$ which induces $t$ palettes and $t\leqslant \delta(G)$. As already proved, we have $|{\cal C}|> \delta(G)$.
Now we prove the existence of a subset of $\cal C$ which induces a spanning even subgraph $K$ of $G$ with $\delta(K)\geqslant 2$, thus obtaining a contradiction. 

Let $A$ be a largest subset of $\cal C$ such that $G_A$ is an even subgraph of $G$. Note that such a subset does exist since the empty-set induces an even subgraph of $G$.

We only need to prove that $\delta(G_A)\geqslant 2$. Assume that this is not the case, then $G_A$ has an isolated vertex $v$. Let $P_j$ be the palette of $v$ in the edge-coloring $c$. Since $d_{G_A}(v)=0$, $P_j \cap A=\emptyset$ holds.

Clearly $|P_j|\geqslant \delta(G)$ holds and we proved that $|{\cal C}|> \delta(G)$. Hence, we can construct a subset $R_j$ of $\cal C$ by choosing $\delta(G)$ arbitrary elements from $P_j$ and an additional color, denoted by $\alpha$, in the following way:
\begin{enumerate}
 \item $\alpha \in P_j$, if $|P_j|>\delta(G)$; 
 \item $\alpha \notin P_j$, if $|P_j|=\delta(G)<|\cal{C}|.$
\end{enumerate}

Note that $\alpha$ could belong to $A$ only in the second case. 

The number of non empty subsets of $R_j$ is $2^{\delta(G)+1} - 1$, while the  possible images of these subsets under the action of $\phi_c$ are at most $2^{t}$, that is the order of the elementary abelian group $\mathbb{Z}_2^t$. Since we assumed $t\leq \delta(G)$, there exist two distinct subsets $I_1$ and $I_2$ of $R_j$ such that $\phi_c(I_1)=\phi_c(I_2)$.

Consider the symmetric difference $I= I_1 \triangle I_2$. The following holds: 
$$\phi_c(I)=\phi_c(I_1 \triangle I_2)= \phi_c(I_1) + \phi_c(I_2)= (0,\dots,0), $$
where $+$ denote the addition in the group $\mathbb{Z}_2^t$.

Finally we obtain a contradiction by proving that the set $A \triangle I$ is a set of colors larger than $A$ which induces an even subgraph of $G$.

If $\alpha \in R_j$, then $A$ and $I$ are disjoint sets and $I$ is not empty, then $A \triangle I = A \cup I$ is larger than $A$.

If $\alpha \notin R_j$, then $I$ contains at least two elements different from $\alpha$. Indeed, it contains at least one of element of $P_j$ and, more precisely, it must contain an even number of them by $\phi_c(I)=(0,\dots,0)$.
Moreover, $A \cap I \subseteq \{\alpha\}$ holds since $\alpha$ is the unique possible element of $A$ which could belong to $I$. Then, $|A \triangle I|=|A|+|I|-2|A \cap I| > |A|$ where the last inequality holds since $|I|>2$ whenever $A \cap I = \{ \alpha \} $. Then, $A \triangle I$ is larger than $A$ again. 
Finally,
$$\phi_c(A \triangle I)=\phi_c(A)+\phi_c(I) = (0,\dots,0),$$ 
since $\phi_c(A)=(0,\dots,0)$ by assumption and $\phi_c(I)=(0,\dots,0)$ as already proved. Then, $A \triangle I$ induces an even subgraph of $G$ and it is larger than $A$, a contradiction again. 
It follows that $G_A$ is a spanning even subgraph of $G$ without isolated vertices and the assertion is proved.
\end{proof}

\section{Regular graphs with palette index as large as possible}
\label{sec:regular}
In this section we construct families of $r$-regular graphs having palette index equal to $r+1$. 

First of all, we deduce the following easy consequence of Theorem \ref{main_theorem} in the case of odd regular graphs.

\begin{cor}\label{cor:regular}
For every positive integer $k$, let $G$ be a $(2k+1)$-regular graph with no spanning even subgraph without isolated vertices. Then, $\check s(G)=2k+2$. 
\end{cor}
\begin{proof}
As already observed, the palette index of a $(2k+1)$-regular graph cannot be larger than $2k+2$. Then, it suffices to prove that $\check s(G)>2k+1$.
The relation $\delta(G)=\Delta(G)=2k+1 > 1$ holds and $G$ has not any spanning even subgraph without isolated vertices. Then $G$ has palette index larger than $2k+1$ by Theorem \ref{main_theorem}. 
\end{proof}

Corollary \ref{cor:regular} gives a sufficient condition for an odd regular graph to have maximum possible palette index. Theorem \ref{cubic} says that this sufficient condition is also necessary in the cubic case (i.e.\ $k=1$). Indeed, when $k=1$, the existence of a spanning even subgraph without isolated vertices is equivalent to the existence of a $2$-factor.

We wonder if the same holds for $k>1$ and we leave it as an open problem.

\begin{prob}
Prove (or disprove) that if $G$ is a $(2k+1)$-regular graph such that $\check s(G)=2k+2$, then $G$ has not a spanning even subgraph without isolated vertices.  
\end{prob}

We complete this section by showing, for every integer $k$, a family of $(2k+1)$-regular graphs which satisfy the condition of Corollary \ref{cor:regular}, and thus having palette index $2k+2$.

It suffices to consider a $(2k+1)$-regular graph $G$ with a vertex $v$ such that every edge incident to $v$ is a bridge of $G$ (see Figure \ref{fig:example}). 
Clearly, $G$ cannot have a spanning even subgraph $K$ with $\delta(K)$ at least $2$, since $K$ should contain at least two of the $2k+1$ bridges incident to $v$, but an even subgraph cannot contain any bridge of $G$.

\begin{figure}[h]
\centering
\includegraphics[width=5cm]{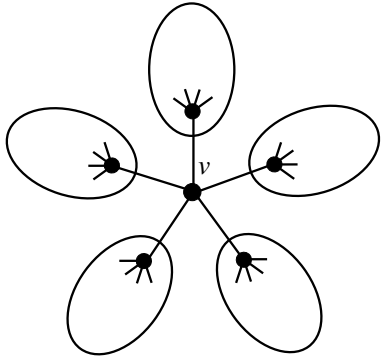} 
\caption{The structure of a $5$-regular graph with palette index $6$.}\label{fig:example}
\end{figure}

Then, we can state the following proposition which gives an answer, for odd regular graphs, to an open problem in \cite{BonMaz}, where Bonvicini and the second author wondered about the existence of a $r$-regular graph with palette index $r+1$ for every $r>4$.  

\begin{prop}\label{max_palette}
For every positive integer $k$, let $G$ be a $(2k+1)$-regular graph admitting a vertex $v$ such that every edge incident to $v$ is a bridge of $G$. Then, $\check s(G)=2k+2$.
\end{prop}

The problem of establishing the existence of a $2k$-regular graph with palette index $2k+1$, for $k>2$, remains open.

\section{Graphs with palette index growing quadratically with respect to their maximum degree}\label{sec:quadratic}

In \cite{Aves_Bonis_Mazz}, the authors present a family of multigraphs whose palette index is expressed by a quadratic polynomial in $\Delta$. They leave the construction of a family of graphs without multiple edges with the same property as an open problem.
We use our main result in previous section to obtain such a family of graphs in a very straightforward way.

For every positive integer $i$, let $G_i$ be a $(2i+1)$-regular graph with $\check s(G_i)=2i+2$ (see Corollary \ref{max_palette}). Moreover, for every positive integer $k$, let $H_k$ be the disjoint union of $G_1, G_2, \dots, G_k$.

Clearly, $\Delta(H_k)=\Delta(G_k)=2k+1$ holds. Moreover, every connected component of $H_k$ has vertices of degree different from the degree of the vertices in every other component. Hence, it follows

$$\check s(H_k)= \sum_{i=1}^{k} \check s(G_i) = \sum_{i=1}^{k} (2i+2) = k^2 + 3k = \frac{\Delta^2+4\Delta-5}{4}.$$

It is not hard to construct examples of graphs with the same property and also connected.
Starting from $H_k$, we add a new extra vertex which is declared to be adjacent to exactly one vertex in each connected component of $H_k$; the choice of the vertex in each component is unrelevant. The graph so obtained is connected and it has maximum degree one more than $H_k$. Moreover, it is an easy check that its palette index is larger than $\check s(H_k)$. Then, we have an infinite family of graphs without multiple edges whose palette index grows quadratically with respect to their maximum degrees.

\end{document}